\documentclass[12pt]{amsart}

 \usepackage{amssymb}

 \usepackage{url}

\newtheorem{theorem}{Theorem}[section]
\newtheorem{corollary}[theorem]{Corollary}
\newtheorem{prop}[theorem]{Proposition}
\newtheorem{lemma}[theorem]{Lemma}

\theoremstyle{remark}

\theoremstyle{definition}

\numberwithin{equation}{section}
\numberwithin{theorem}{section}

\begin{document}




\date{Preprint May 16, 2012.  To appear in {\it Journal of Mathematical
Inequalities}}


\title{Optimal error estimates for corrected trapezoidal rules}


%
\author{Erik Talvila}
\address{Department of Mathematics \& Statistics\\
University of the Fraser Valley\\
Abbotsford, BC Canada V2S 7M8}
\email{Erik.Talvila@ufv.ca}

\author{Matthew Wiersma}
\address{Department of Pure Mathematics\\
University of Waterloo\\
Waterloo, ON Canada N2L 3G1}
\email{mwiersma@uwaterloo.ca}
\thanks{The first author was supported by a Discovery Grant,
the second author was supported by an Undergraduate Student Research Award 
held at University of the Fraser Valley; 
both from the
Natural Sciences and Engineering Research Council of Canada.
This paper was written while the
first author was on leave and visiting the Department of Mathematical and 
Statistical
Sciences, University of Alberta.}


%

%

\begin{abstract}
Corrected trapezoidal rules are proved for $\int_a^b f(x)\,dx$ under the
assumption that $f''\in L^p([a,b])$ for some $1\leq p\leq\infty$.  Such
quadrature rules involve the trapezoidal rule modified by the addition
of a term $k[f'(a)-f'(b)]$.  The coefficient  $k$ in the
quadrature formula is found that minimizes the error estimates.  It is shown
that when $f'$ is merely assumed to be continuous then the optimal rule
is the trapezoidal rule itself.  In this case error estimates are in
terms of the Alexiewicz norm.  This includes the case when $f''$ is
integrable in the Henstock--Kurzweil sense or as a distribution.  All
error estimates are shown to be sharp for the given assumptions on $f''$.
It is shown how to make these formulas exact for all cubic polynomials $f$. 
Composite formulas are computed for uniform partitions.
\end{abstract}

\maketitle

\let\oldsqrt\sqrt
\def\sqrt{\mathpalette\DHLhksqrt}
\def\DHLhksqrt#1#2{%
\setbox0=\hbox{$#1\oldsqrt{#2\,}$}\dimen0=\ht0
\advance\dimen0-0.2\ht0
\setbox2=\hbox{\vrule height\ht0 depth -\dimen0}%
{\box0\lower0.4pt\box2}}

\newcommand{\alexc}{{\mathcal A}_c([a,b])}
\newcommand{\alexr}{{\mathcal A}_r([a,b])}
\newcommand{\balexc}{{\mathcal B}_c([a,b])}
\newcommand{\balexr}{{\mathcal B}_r([a,b])}
\newcommand{\fn}{\!:\!}
\newcommand{\bv}{{\mathcal BV}}
\newcommand{\be}{\begin{equation}}
\newcommand{\ee}{\end{equation}}
\providecommand{\abs}[1]{\lvert#1\rvert}
\providecommand{\norm}[1]{\lVert#1\rVert}
\newcommand{\intab}{\int_a^b}
\newcommand{\polym}{{\mathcal P}_m}
\newcommand{\poly}[1]{{\mathcal P}_{#1}}
\newcommand{\R}{{\mathbb R}}
\newcommand{\N}{{\mathbb N}}
\section{Introduction}
This paper is concerned with numerical integration schemes
for $\intab f(x)\,dx$ where it is assumed $f'$ is absolutely
continuous such that $f''\in L^p([a,b])$ for some $1\leq p\leq\infty$,
or that $f''$ is integrable in the Henstock--Kurzweil sense, or that
$f'$ is continuous so that $f''$ exists as a distribution and is
integrable using a distributional integral.
Integration by parts shows that
\begin{equation}
\intab f(x)\,dx=\frac{1}{2}\left[-f(a)\phi'(a)+f(b)\phi'(b)+
f'(a)\phi(a)-f'(b)\phi(b)\right]
+E(f),\label{basic}
\end{equation}
where $E(f)=(1/2)\intab f''(x)\phi(x)\,dx$ and $\phi$ is a 
monic quadratic
polynomial.  Observe that taking $\phi(x)=(x-a)(x-b)$ gives the
usual trapezoidal rule
$\intab f(x)\,dx=(b-a)\left[f(a)+f(b)\right]/2
+E^{T}(f)$,
where $E^{T}(f)=(1/2)\intab f''(x)(x-a)(x-b)\,dx$.
The H\"older inequality then
gives $\abs{E^T(f)}\leq \norm{f''}_p\norm{\phi}_q/2$, where $q$ is
the conjugate exponent of $p$.  Hence, 
\begin{equation}
\abs{E^T(f)}  \leq  \left\{
\begin{array}{ll}
\norm{f''}_1(b-a)^{2}/8, & p=1\\
{[B(q+1,q+1)]}^{1/q}\norm{f''}_p(b-a)^{2+1/q}/2, & 1<p<\infty\\
\norm{f''}_\infty(b-a)^{3}/12, & p=\infty.
\end{array}
\right.\label{trapp}
\end{equation}
Here, $B(x,y)=\Gamma(x)\Gamma(y)/\Gamma(x+y)$ is the beta function.
See
\cite[Theorem~3.22]{ceronedragomirtrap}.  (This corrects a
typographical error in \cite{ceronedragomirtrap}.)

We find $\phi$ that minimizes the error in \eqref{basic}.  This
leads to a quadrature rule that includes values of $f$ and $f'$ at
the endpoints $a$ and $b$.
The error-minimizing polynomial produces the
classical trapezoidal rule, modified by the addition of first
derivative terms.  In the literature this is a called a 
corrected trapezoidal rule.
This
includes solving the problem of choosing $k\in\R$ to minimize all
quadrature rules of the form $(b-a)\left[f(a)+f(b)\right]/2 +k[f'(a)-
f'(b)]$.
The error terms are as in \eqref{trapp} but with
the coefficient of $\norm{f''}_p(b-a)^{2+1/q}$ minimized.
In particular, the coefficients are strictly less than
in \eqref{trapp}.
We prove our results are the best possible given the
assumption $f''\in L^p([a,b])$ (Theorem~\ref{theoremminimize} and
Corollaries~\ref{corollary1}-\ref{corollarycomposite}).
The composite formula (Corollary~\ref{corollarycomposite}) provides an
improved error estimate over the trapezoidal rule.  Since the $f'$ terms 
telescope, the correction terms only require computation of $f'$ at
endpoints $a$ and $b$ rather than at interior nodes.
Compared to the usual trapezoidal rule the extra computing time to implement our
corrected rule is then negligible for large $n$.

Finding the polynomial $\phi$ that minimizes the error in
\eqref{basic} involves solving a transcendental equation for
a parameter in the polynomial.  This transcendental equation
is written in various ways
in Section~\ref{sectionbeta}.  When $q$ is an integer this
reduces to a polynomial equation for the parameter.  We are
able to solve for the exact value of the parameter when
$p=1,2,4/3, \infty$.  See Corollaries~\ref{corollary1}-\ref{corollaryinfinity}
and \eqref{3/2}. Even without knowing the parameter exactly,
Theorem~\ref{theoremapproxbeta} gives
a corrected trapezoidal rule with error estimate smaller
than in \eqref{trapp}.

In Section~\ref{sectionhk} we reduce the assumption on $f$ to
$f\in C^1([a,b])$ and then $f''$ exists as a distribution
and is integrable using the continuous primitive integral
(Corollary~\ref{corollaryalexc}).  This includes the case when
$f''$ is integrable in the
Henstock--Kurzweil sense (Theorem~\ref{theoremhk}).
The error estimate is then in terms of the Alexiewicz norm of $f''$.
In this case, the optimum form of \eqref{basic} is the trapezoidal
rule itself.

In Section~\ref{sectionexact} we compute $\phi$ so that \eqref{basic} 
is exact for all cubic polynomials $f$.  The required $\phi$ is the
same as the one that minimizes the error in the case when $f''\in L^2([a,b])$.

Several authors have considered corrected trapezoidal rules under the
assumption that the derivatives of $f$ are in various function spaces.  
See \cite{ceronedragomirtrap} ($L^p$), \cite{davisrabinowitz} ($L^\infty$),
\cite{dedicmaticpecaric} (Lipschitz, continuous and bounded variation,
$L^p$), \cite{dingyewang} ($L^p$, Henstock--Kurzweil integrable), 
\cite{liu} (continuous and bounded variation) and
\cite{pecaricujevic} ($L^p$).

\section{Main theorem}
The error in \eqref{basic} is minimized over all monic polynomials
$\phi$.  Results in Lemma~\ref{lemma} show that
a unique error-minimizing polynomial exists and is of the form
$\phi(x)=(x-c)^2-\alpha^2$ such that $c$ is the midpoint of
$[a,b]$ and $\phi$ has two real roots
in $[a,b]$.  Note that \eqref{basic} becomes a corrected
trapezoidal rule precisely when $\phi'(b)=-\phi'(a)=b-a$ and this
relation holds for $\phi(x)=(x-c)^2-\alpha^2$.  For a uniform partition
the composite rule obtained from \eqref{basic} will in general have
$f'$ evaluated at all points at which $f$ is evaluated.  However,
when $\phi'(a)=\phi'(b)$ the sum of $f'$ terms telescopes, leaving
only $f'(a)$ and $f'(b)$ (Corollary~\ref{corollarycomposite}).  This
is the case with all the error-minimizing rules we present.

We are able to compute exact values for $\alpha$
when $p=1,2,4/3,\infty$.  In other cases $\alpha$ is given
by the transcendental equation \eqref{beta}. When $q$ is an integer this
reduces to finding the largest real root of a polynomial of
degree $2q-1$.
\begin{theorem}\label{theoremminimize}
Let $c$ be the midpoint of $[a,b]$.
Let $f\fn[a,b]\to\R$ such that $f'$ is absolutely
continuous and $f''\in L^p([a,b])$ for some $1< p<\infty$.
Let $1/p+1/q=1$.
Amongst all monic quadratic polynomials used to generate \eqref{basic},
taking $\phi(x)=(x-c)^2-\alpha_q^2$ gives the unique minimum for the 
error $\abs{E(f)}$.  The
constant $\beta_q>1$ is the unique solution of the equation
\begin{equation}
\int_1^{\beta_q}
(x^2-1)^{q-1}\,dx =\frac{1}{2}B(q,1/2) = 2^{2q-2}B(q,q)\label{betaeqn}
\end{equation}
and $\alpha_q\beta_q=(b-a)/2$.
This gives the
quadrature formula
\begin{equation}
\intab f(x)\,dx=\frac{b-a}{2}\left[f(a)+f(b)\right] +
\frac{(b-a)^2}{8}(1-\beta_q^{-2})
[f'(a) -f'(b)]
+E(f),\label{optimal}
\end{equation}
where
\begin{equation}
\abs{E(f)}  \leq 
\frac{\norm{f''}_p(b-a)^{2+1/q}(1-\beta_q^{-2})}{2^{3+1/q}(q+1/2)^{1/q}}.
\label{optimalerror}
\end{equation}
The coefficient of $\norm{f''}_p$ in the error bound is the best 
possible.
\end{theorem}
Note that the numbers $\alpha_q$ and $\beta_q$ are independent of
$f$, while $\beta_q$ are also independent of the interval $[a,b]$. 
\begin{corollary}\label{corollary1}
If $p=1$ and  $q=\infty$
then $\beta_\infty=\sqrt{2}$, $\alpha_\infty=(b-a)/(2\sqrt{2})$
gives the unique minimum error.
The quadrature
formula is
\begin{equation}
\intab f(x)\,dx=\frac{(b-a)}{2}\left[f(a)+f(b)\right] +
\frac{(b-a)^2}{16}\left[f'(a) -f'(b)\right]
+E(f),\label{optimal1}
\end{equation}
where the optimal error is
$\abs{E(f)}  \leq 
\norm{f''}_1(b-a)^{2}/16.$
\end{corollary}
\begin{corollary}\label{corollary2}
If $p=q=2$
then $\beta_2=\sqrt{3}$, $\alpha_2=(b-a)/(2\sqrt{3})$ gives the unique minimum error.
The quadrature
formula is
\begin{equation}
\intab f(x)\,dx=\frac{(b-a)}{2}\left[f(a)+f(b)\right] +
\frac{(b-a)^2}{12}\left[f'(a) -f'(b)\right]
+E(f),\label{optimal2}
\end{equation}
where the optimal error is
$\abs{E(f)}  \leq 
\norm{f''}_2(b-a)^{2.5}/(12\sqrt{5}).$
\end{corollary}

\begin{corollary}\label{corollaryinfinity}
If $p=\infty$ and $q=1$
then $\beta_1=2$, $\alpha_1=(b-a)/4$
gives the unique minimum error.
The quadrature
formula is
\begin{equation}
\intab f(x)\,dx=\frac{(b-a)}{2}\left[f(a)+f(b)\right] +
\frac{3(b-a)^2}{32}[f'(a) -f'(b)]
+E(f),\label{optimalinfinity}
\end{equation}
where the optimal error is
$\abs{E(f)}  \leq 
\norm{f''}_\infty(b-a)^{3}/32.$
\end{corollary}

Other authors have obtained corrected trapezoidal rules under the
assumption $f''\in L^p([a,b])$, generally
with different coefficients of $f'(a)-f'(b)$ in \eqref{optimal}
and strictly larger coefficients of
$\norm{f''}_p(b-a)^{2+1/q}$ than in \eqref{optimalerror}.  See
Cerone and Dragomir
\cite{ceronedragomirtrap} equation (3.64), where the coefficient
of  $f'(a)-f'(b)$ is $(b-a)^2/8$ for all values of $p$.  
In Theorem~3.24 for $p=\infty$
they have the coefficient of $f'(a)-f'(b)$ as $(b-a)^2/12$.  This
coefficient is only sharp for $p=2$ (Corollary~\ref{corollary2}).
The
estimate $\abs{E(f)}\leq [\sup(f)-\inf(f)](b-a)^3/
(24\sqrt{5})$ is proved.
Dedi\`{c}, Mati\`{c} and Pe\v{c}ari\`{c}
\cite[Corollaries~9, 12]{dedicmaticpecaric} also consider
corrected trapezoidal rules with $f''\in L^p([a,b])$.
In their Corollary~9 they have the coefficient of $f'(a)-f'(b)$ as 
$(b-a)^2/12$ and
prove the larger estimate $\abs{E(f)}\leq \norm{f''}_\infty(b-a)^3/
(18\sqrt{3})$.
The results of our Corollary~\ref{corollary2} appear as their Corollary~12 
(without sharpness).  Similarly in \cite{dingyewang}.
\begin{corollary}\label{corollarycomposite}
For a uniform partition given by $x_i=a+(b-a)i/n$, $0\leq i\leq n$,
the composite formula is
\begin{equation}
\intab f(x)\,dx=\frac{b-a}{2n}\left[f(a)+2\sum_{i=1}^{n-1} f(x_i)+f(b)\right] +
\frac{(b-a)^2}{8n^2}(1-\beta_q^{-2})
[f'(a) -f'(b)]
+E(f),\label{optimalcomposite}
\end{equation}
where
\begin{eqnarray*}
\abs{E(f)}  \leq \left\{\begin{array}{ll}
\norm{f''}_1(b-a)^{2}/(16n^2), & p=1\\
\frac{\norm{f''}_p(b-a)^{2+1/q}(1-\beta_q^{-2})}{2^{3+1/q}(q+1/2)^{1/q}\,n^{
2}},
 & 1<p<\infty\\
\norm{f''}_\infty(b-a)^{3}/(32n^2), & p=\infty.
\end{array}
\right.
\end{eqnarray*}
The coefficient of $\norm{f''}_p$ in the error bound is the best 
possible.
\end{corollary}
\begin{proof}
[Proof of Theorem~\ref{theoremminimize}] 
Let $\phi$ be a monic quadratic polynomial.
Integration by parts
gives \eqref{basic}.  We are then led to minimize
$E(f)=(1/2)\intab f''(x)\phi(x)\,dx$.  By the H\"older inequality,
$\abs{E(f)}\leq (1/2)\norm{f''}_p\norm{\phi}_q$.  First consider a symmetric
interval $[-a,a]$.  By Lemma~\ref{lemma}, to 
minimize $\norm{\phi}_q$ we
need only consider $\phi(x)=x^2-\alpha^2$ for some $\alpha\in[0,a]$.
Let $\beta=a/\alpha$.
Define 
\begin{equation}
G_q(\alpha)  =  \norm{\phi}_q^q=2\int_0^a \abs{x^2-\alpha^2}^q\,dx\\
  =  2\alpha^{2q+1}\left(\int_0^1(1-x^2)^q\,dx+\int_1^{a/\alpha}
(x^2-1)^q\,dx\right).\label{Gq}
\end{equation}
Note that
\begin{equation}
G'_q(\alpha)  = 4q\alpha^{2q}\left(\int_0^1(1-x^2)^{q-1}\,dx-\int_1^{a/\alpha}
(x^2-1)^{q-1}\,dx\right).\label{G'q}
\end{equation}
We have $G_q'(0)=0$.  Since the function $\alpha\mapsto \int_1^{a/\alpha}
(x^2-1)^{q-1}\,dx$ decreases from positive infinity to zero as 
$\alpha$ increases from zero to $a$,
it follows that $G_q'$ has a unique root in $(0,a)$ and $G_q$ has a
unique minimum at $\alpha_q\in(0,a)$.
The first integral in \eqref{G'q} can be evaluated in terms of beta and
gamma functions.  We have \cite[5.12.1]{nist}
$$
\int_0^1(1-x^2)^{q-1}\,dx =\frac{1}{2}B(q,1/2)=2^{2q-2}B(q,q)=
\frac{\sqrt{\pi}\,\Gamma(q)}{2\Gamma(q+1/2)}.
$$
Let $\beta_q=a/\alpha_q$.
The required minimizing polynomial is then determined by the
unique root $\beta_q\in(1,\infty)$ of the equation
\begin{equation}
\int_1^{\beta_q}
(x^2-1)^{q-1}\,dx =\frac{1}{2}B(q,1/2).\label{beta}
\end{equation}

To compute $\norm{\phi}_q$, evaluate the final integral in \eqref{Gq}. 
Integration by parts establishes the recurrence relation
\begin{equation}
\int(x^2-1)^q\,dx=\frac{x(x^2-1)^q}{2q+1}-\frac{2q}{2q+1}\int(x^2-1)^{q-1}\,dx.
\label{recurrence}
\end{equation}
Using this and the corresponding version with integrand $(1-x^2)^q$, we
obtain $G_q(\alpha_q)=2a(a^2-\alpha_q^2)^q/(2q+1)$.
It then follows that 
$$
\abs{E(f)}\leq 
\frac{\norm{f''}_p(2a)^{2+1/q}[1-(\alpha_q/a)^2]}{2^{3+1/q}(q+1/2)^{1/q}}.
$$
Replacing $a$ with $(b-a)/2$ establishes the error estimate for interval
$[a,b]$.  Using $\phi(x)=(x-c)^2-\alpha_q^2$ the formula \eqref{optimal}
now follows.

With the H\"older inequality,
$\abs{\intab f''(x)\phi(x)\,dx}\leq\norm{f''}_p\norm{\phi}_q$,
the necessary and sufficient condition for equality
is $f''(x)=d\,{\rm sgn}[\phi(x)]\abs{\phi(x)}^{1/(p-1)}$ for some $d\in\R$ and 
almost
all $x\in[a,b]$.  See \cite[p.~46]{liebloss}.
Integrating gives $f(x)=d\int_a^x(x-t){\rm sgn}[\phi(t)]\abs{\phi(t)}^{1/(p-1)}\,
dt$, modulo a linear function.
\end{proof}

\begin{proof}
[Proof of Corollary~\ref{corollary1}]
It suffices to consider the interval $[-a,a]$.
When $p=1$ write $\phi(x)=x^2-\alpha_\infty^2$.  
By Lemma~\ref{lemma} we need only
consider the case with two distinct roots in $(-a,a)$, i.e., 
$0< \alpha_\infty<a$. 
We have
$$
F_\infty(\phi):=\norm{\phi}_\infty=
\max_{\abs{x}\leq a}\abs{\phi(x)}=\max(\abs{\phi(0)},\phi(a))
=\max(\alpha_\infty^2,a^2-\alpha_\infty^2).
$$
It follows that
$\alpha_\infty=a/\sqrt{2}$. 
Formula \eqref{optimal1} and the error estimate now follow easily.
There is equality in 
$\abs{\int_{-a}^a f''(x)\phi(x)\,dx}\leq\norm{f''}_1\norm{\phi}_\infty$ whenever
$\phi(x)=d\,{\rm sgn}[f''(x)]$ for some $d\in\R$ and almost
all $x\in[-a,a]$.  See \cite[p.~46]{liebloss}.
This does not hold but we can show the coefficient of $\norm{f''}_1$
cannot be reduced by considering a sequence of functions.  If
$f''=\delta$, the Dirac distribution supported at $0$, then 
$\abs{\int_{-a}^a f''(x)(x^2-\alpha_\infty^2)\,dx}= \alpha_\infty^2=a^2/2
=\norm{f''}_1\norm{\phi}_\infty$.  Of course, if $f''=\delta$
then $f'=\chi_{(0,\infty)}$ which is not absolutely continuous.  
Let
$(\psi_n)$ be a delta sequence.  This is a sequence
of continuous functions $\psi_n\geq 0$ with support in $(0,1/n)$ such
that $\int_0^{1/n}\psi_n(x)\,dx=1$.  Now define 
$f_n(x)=\int_{-a}^x\int_{-a}^y\psi_n(z)\,dz\,dy$.  Then 
$f_n'$ is absolutely continuous and $\norm{f_n''}_1=1$.   
Since $\phi$ is continuous we
have $\lim_{n\to\infty}\abs{\int_{-a}^a f_n''(x)(x^2-\alpha_\infty^2)\,dx}=a^2/2
=\norm{\phi}_\infty$.
The error estimate is then optimal.
\end{proof}
\begin{proof}
[Proof of Corollary~\ref{corollaryinfinity}]
Write $\phi(x)=x^2-\alpha_1^2$.  By Lemma~\ref{lemma} we need only
consider the case with two distinct roots in $(-a,a)$, i.e., 
$0< \alpha_1< a$.  Equation \eqref{G'q} now becomes
$G_1'(\alpha_1)=2\alpha_1(12\alpha_1-6a)/3$, from which $\alpha_1=a/2$.
This then gives \eqref{optimalinfinity}.
There is equality in 
$\abs{\int_{-a}^a f''(x)\phi(x)\,dx}\leq\norm{f''}_\infty\norm{\phi}_1$ whenever
$f''(x)=d\,{\rm sgn}[\phi(x)]$ for some $d\in\R$ and almost
all $x\in[-a,a]$.  See \cite[p.~46]{liebloss}.
Integrating gives $f(x)=d\int_{-a}^x(x-t){\rm sgn}[\phi(t)]\,
dt$, modulo a linear function.
\end{proof}
This case also appears in \cite[Theorem~1]{talvilawiersma}.

The proof of Corollary~\ref{corollarycomposite}
follows using the H\"older inequality for series as in the proof
of Theorem~3.26 in \cite{ceronedragomirtrap}.

Lemma~\ref{lemma}
shows that the
minimum of $\norm{\phi}_q$ over monic quadratics is unique.  
Hence, the coefficient
of $\norm{f''}_p(b-a)^{2+1/q}$ in \eqref{optimalerror} is
strictly less than for any other choice of $\alpha_q$ and indeed
for any other choice of monic quadratic.  In particular, we
get a smaller coefficient than in the trapezoidal rule \eqref{trapp}.

\section{Evaluation and approximation of $\beta_q$}\label{sectionbeta}
In Corollaries~\ref{corollary1} through \ref{corollaryinfinity} we
were able to compute the exact value of $\alpha_q$ and $\beta_q$
for $q=1,2,\infty$.  In this section we compute $\beta_4$ as the
root of a cubic polynomial.  When $q$ is an integer, equation
\eqref{beta} becomes a polynomial of degree $2q-1$.  See
\eqref{poly1} and \eqref{poly2}.  When $q$ is
even the degree reduces to $q-1$ and we compute the exact value
of $\beta_4$ \eqref{3/2}.  In general, equation \eqref{beta}
is transcendental and most likely cannot be solved exactly.  We rewrite this in terms of hypergeometric
and associated Legendre functions and also show that 
$\sqrt{2}\leq\beta_q\leq 2$ and is decreasing 
(Proposition~\ref{betadecreasing}).    In Theorem~\ref{theoremapproxbeta}
we show how
the corrected trapezoidal rule can give good approximations 
of the integral of $f$ with $f''\in L^p([a,b])$ for all $1\leq p\leq\infty$
even if the exact value of $\beta_q$ is not known.  Part (c) of
this theorem shows that the corrected trapezoidal rule with $\alpha=0$,
$\beta=\infty$ gives a smaller error estimate than \eqref{trapp}
for all $1\leq q<\infty$.  The coefficient is a simple function of $q$.

If $q$ is
an integer, use the binomial theorem to write \eqref{beta} as
\begin{equation}
\int_1^{\beta_q}
(x^2-1)^{q-1}\,dx=\sum_{k=0}^{q-1}\binom{q-1}{k}\frac{(-1)^{q-1+k}}{2k+1}
\left[\beta_q^{2k+1}-1\right]
=\sum_{k=0}^{q-1}\binom{q-1}{k}\frac{(-1)^{k}}{2k+1}.\label{poly1}
\end{equation}
This gives a polynomial of degree $2q-1$ for $\beta_q$.  
When $q$ is
even this reduces to 
\begin{equation}
\sum_{k=0}^{q-1}\binom{q-1}{k}\frac{(-1)^{k}\beta_q^{2k}}{2k+1}=0.
\label{betapoly1}
\end{equation}

When $p=4/3$ and $q=4$ we get the polynomial 
$\beta_4^6-(21/5)\beta_4^4+7\beta_4^2-7=0$.
The unique solution is \cite[1.11(iii)]{nist} 
\begin{equation}
\beta_4=\left\{
\frac{2(7^{1/3})}{(3^{2/3})5}\left(\left[5\sqrt{30}+27\right]^{1/3}-
\left[5\sqrt{30}-27\right]^{1/3}\right)+7/5\right\}^{1/2}
\doteq 1.589291662.
\label{3/2}
\end{equation}
This can then be used in equations \eqref{optimal} and
\eqref{optimalerror}.

Repeated use of \eqref{recurrence} yields the equivalent series
form of \eqref{beta} when $q$ is  an integer
\begin{equation}
\beta_q\sum_{k=0}^{q-1}(-1)^k
\binom{2k}{k}\left(\frac{\beta_q^2-1}{4}\right)^k=1-(-1)^q.\label{poly2}
\end{equation}
When $q$ is even this simplifies to
\begin{equation}
\sum_{k=0}^{q-1}\binom{2k}{k}\left(\frac{1-\beta_q^2}{4}\right)^k
=\sum_{k=0}^{q-1}\frac{(1/2)_k}{k!}(1-\beta_q^2)^k=0.
\end{equation}

When $p=3$ and $q=3/2$ the integrals in \eqref{beta} can be
evaluated in terms of elementary functions but this leads to a
transcendental equation for $\beta_{3/2}$. Similarly when $q$ is a
half integer.  

As can be seen from \eqref{3/2} the numbers $\beta_q$ are not 
necessarily simple functions of $q$.  For cases other than
$p=1,2,\infty$ they can be numerically approximated and for this
there are many other ways \eqref{beta} can be rewritten.  For example,
using a linear change of variables and then the identities 
\cite[15.6.1, 15.9.21]{nist} the integral in \eqref{beta}
can be written in terms of hypergeometric and associated Legendre
functions.  The result is
\begin{eqnarray*}
\int_1^{\beta_q}
(x^2-1)^{q-1}\,dx & = &
2^{q-1}(\beta_q-1)^q\int_0^1\left[1-(1-\beta_q)x/2\right]^{q-1}x^{q-1}\,dx\\
 & = & \frac{2^{q-1}(\beta_q-1)^q}{q} {_2}F_1(1-q,q;1+q;(1-\beta_q)/2)\\
 & = & 2^{q-1}\Gamma(q)\left(\beta_q^2-1\right)^{q/2}
P_{q-1}^{-q}(\beta_q).
\end{eqnarray*}

Numerical equation solvers can now be applied to any of these
representations of $\int_1^{\beta_q}
(x^2-1)^{q-1}\,dx$ in order to solve \eqref{beta}.

The range of $\beta_q$ is known.
\begin{prop}\label{betadecreasing}
$\beta_q$ is a decreasing function of $q$ and 
$\sqrt{2}\leq \beta_q\leq 2$ for $1\leq q\leq\infty$.
\end{prop}

\begin{proof}
A change of variables shows that \eqref{betaeqn} is equivalent to
$\int_0^{\beta_q^2-1}y^{q-1}(1+y)^{-1/2}\,dy=J(q)$, where 
$J(q)=\int_0^1 y^{q-1}(1-y)^{-1/2}\,dy$.
Since $(1+y)^{-1/2}\leq (1-y)^{-1/2}$ for
all $0\leq y<1$ we must have $\beta_q^2-1\geq 1$.

Equation \eqref{betaeqn} can also be written as
$
I(q)+2\int_{\sqrt{2}}^{\beta_q}(x^2-1)^{q-1}\,dx
=J(q)$
where $I(q)=\int_0^1y^{q-1}(1+y)^{-1/2}\,dy$.
The argument above shows that $J'(q)<I'(q)<0$ for all $1<q<\infty$.
Since $\beta_q\geq\sqrt{2}$ it must be a decreasing function of $q$.
\end{proof}
Even without knowing the exact value of $\beta_q$ we can use the
method of Theorem~\ref{theoremminimize} to obtain good estimates
of the error in corrected trapezoidal rules.
\begin{theorem}\label{theoremapproxbeta}
Let $c$ be the midpoint of $[a,b]$.
Let $f\fn[a,b]\to\R$ such that $f'$ is absolutely
continuous and $f''\in L^p([a,b])$ for some $1\leq p<\infty$.
(a) If $1\leq p<2$ let $\phi(x)=(x-c)^2-(b-a)^2/8$.
Equation \eqref{basic} gives
the quadrature
formula
\begin{equation}
\intab f(x)\,dx=\frac{(b-a)}{2}\left[f(a)+f(b)\right] +
\frac{(b-a)^2}{16}\left[f'(a) -f'(b)\right]
+E(f),\label{approxbeta1}
\end{equation}
where the error satisfies
$\abs{E(f)}  \leq 
\norm{f''}_1(b-a)^{2}/16.$
(b) If $2\leq p<\infty$ let $\phi(x)=(x-c)^2-(b-a)^2/12$.
Equation \eqref{basic} gives
the quadrature
formula
\begin{equation}
\intab f(x)\,dx=\frac{(b-a)}{2}\left[f(a)+f(b)\right] +
\frac{(b-a)^2}{12}\left[f'(a) -f'(b)\right]
+E(f),\label{approxbeta2}
\end{equation}
where the error satisfies
$\abs{E(f)}  \leq 
\norm{f''}_2(b-a)^{2.5}/(12\sqrt{5}).$
(c) If $1< p<\infty$ let $\phi(x)=(x-c)^2$ and $1/p+1/q=1$.
Equation \eqref{basic} gives
the quadrature
formula
\begin{equation}
\intab f(x)\,dx=\frac{(b-a)}{2}\left[f(a)+f(b)\right] +
\frac{(b-a)^2}{8}\left[f'(a) -f'(b)\right]
+E(f),\label{approxbeta1infinity}
\end{equation}
where the error satisfies
\begin{equation}
\abs{E(f)}  \leq 
\frac{\norm{f''}_p(b-a)^{2+1/q}}{2^{3+1/q}(q+1/2)^{1/q}}.
\end{equation}
\end{theorem}
\begin{proof}
Use the fact that $L^s([a,b])\subset L^r([a,b])$ whenever
$1\leq r\leq s\leq\infty$.  In (a) use the approximation
from Corollary~\ref{corollary1} and in (b) use the approximation
from Corollary~\ref{corollary2}.  In (c) take $\alpha=0$, 
$\beta=\infty$ and
then compute
$$
\norm{\phi}_q^q=2\left(\frac{b-a}{2}\right)^{2q+1}\int_0^1x^{2q}\,dx=
\frac{(b-a)^{2q+1}}{2^{2q+1}(q+1/2)}.
$$
The rest follows as in the proof of Theorem~\ref{theoremminimize}.
\end{proof}
It is also possible to compute $\norm{\phi}_q$ when $\beta=1$ but this gives
the trapezoidal rule \eqref{trapp}.  Note that the coefficient in
part (c) is strictly less than the corresponding coefficient in \eqref{trapp}
for all $1\leq q<\infty$.  In the limit $q\to\infty$ the coefficient becomes
$1/8$ as in \eqref{trapp}.

\section{Lemmas}
Let $\polym$ be the set of monic polynomials of degree $m\in\N$ with real
coefficients.
Define $F_q\fn\polym\to\R$ by $F_q(\phi)=\norm{\phi}_q$ where 
$1\leq q\leq \infty$
and the norms are over compact interval $[a,b]$.
Since $F_q(\phi)$ is bounded below for $\phi\in\polym$ it has
an infimum over $\polym$.  It also has a unique minimum at a polynomial that
has $m$ roots in $[a,b]$.  As well, the error-minimizing polynomial is
even or odd about the midpoint of $[a,b]$ as $m$ is even or odd.

\begin{lemma}\label{lemma}
(a) For $m\geq 2$, let $\phi\in\polym$ with a non-real root.  There  exists
$\psi\in\polym$ with a real root such that $F_q(\psi)<F_q(\phi)$.
(b) Let $\phi\in\polym$ with a root $t\not\in[a,b]$.  
There  exists $\psi\in\polym$ with a root in $[a,b]$ such that 
$F_q(\psi)<F_q(\phi)$.
(c) If $\phi$ minimizes $F_q$ then it has $m$ simple zeros
in $[a,b]$.
(d) If $F_q$ has a minimum in $\polym$ it is unique.
(e) $F_q$ attains its minimum over $\polym$.
(f) If $\phi\in\polym$ is neither even nor odd about $c:=(a+b)/2$ then
there is a polynomial $\psi\in\polym$ that is either even or odd
about $c$ such that $F_q(\psi)< F_q(\phi)$.
(g) The minimum of $F_q$ occurs
at a polynomial $\phi\in\polym$ with $m$ simple zeros
in $[a,b]$.  If $m$ is even about $c$ then
so is $\phi$.  If $m$ is odd about $c$ then so is $\phi$.
This minimizing polynomial is unique.
\end{lemma}
Theorem~\ref{theoremminimize} and its corollaries used only
$\poly{2}$ but the lemmas give
results for $\polym$ for all $m$.  These will be useful for
considering integration schemes generated by polynomials of
degree $m$, which we do not include here.

The results of Lemma~\ref{lemma} are well known and go
back to Chebyshev and Fej\'{e}r. 
To keep the paper self contained we have provided elementary proofs.
For a full exposition and references
to the original literature see, for example,
\cite{cheney}, \cite{davis} and \cite{lorentz}.  
Three cases of the minimizing problem
of Lemma~\ref{lemma} appear in the literature.  When $q=\infty$ the
minimizing polynomial in ${\mathcal P}_2$ is the Chebyshev polynomial
of the first kind $\phi(x)=x^2-1/2=T_2(x)/2$.  When $q=2$ it
is given as a Legendre polynomial $\phi(x)=x^2-1/3=(2/3)P_2(x)$.
When $q=1$ it is  the Chebyshev polynomial
of the  second kind $\phi(x)=x^2-1/4=U_2(x)/4$.  These are for
the interval $[-1,1]$.  A linear change of variables is used for
other intervals.  These are all types of Gegenbauer polynomials, which
are orthogonal on $[-1,1]$ with respect to a certain weight function.  
Gillis and Lewis \cite{gillislewis} give an argument to show that the
minimizing polynomials are most likely not orthogonal polynomials for
other values of $q$.

\begin{proof}
(a) Write $\phi(x)=[(x-r)^2+s^2]\omega(x)$ for some $r,s\in\R$,
$s\not=0$, and $\omega\in\poly{m-2}$.
Let $\psi(x)=(x-r)^2\omega(x)$.  For all $x\in\R$ such that $\omega(x)\not=0$
we have
$\abs{\psi(x)}< \abs{\phi(x)}$.  Hence, $\norm{\psi}_q<
\norm{\phi}_q$.

(b)  If $m=1$ let $\phi(x)=x-t$. Direct calculation shows the unique minimum of 
$\norm{\phi}_q$ occurs at $t=c$.  If $m\geq 2$, 
by (a) we can assume all the roots of $\phi$ are real and
write $\phi(x)=(x-t)\omega(x)$ for some $t\not\in[a,b]$,
and $\omega\in\poly{m-1}$.  Suppose $t<a$.  Let $\psi(x)=(x-a)\omega(x)$.
For all $x\in[a,b]$ such that $\omega(x)\not=0$ we have $\abs{\psi(x)}<\abs{\phi(x)}$.
Hence,
$\norm{\psi}_q<\norm{\phi}_q$.  Similarly if $t>b$.

(c) Consider $\psi(x)=(x-t)^2$ with 
$t\in(a,b)$ and
$\psi_\epsilon(x)=(x-t+\epsilon)(x-t-\epsilon)=\psi(x)
-\epsilon^2$.  For all $x$ we have $\abs{\psi_\epsilon(x)}\leq
\abs{\psi(x)}+\epsilon^2$ and for $x\not\in(t-\epsilon,t+\epsilon)$
we have $\abs{\psi_\epsilon(x)}<
\abs{\psi(x)}$.  This shows that 
$\norm{\psi_\epsilon}_\infty<\norm{\psi}_\infty$ if $\epsilon>0$ is small
enough.  Factoring now shows the zeros of any minimizing polynomial
are simple.  Similarly for
$t=a$ or $b$.  

For $q=1$,
\begin{eqnarray*}
\norm{\psi_\epsilon}_1 & = & \int_{x\in(t-\epsilon,t+\epsilon)}
\left[\epsilon^2-\psi(x)\right]\,dx+
\int_{x\not\in(t-\epsilon,t+\epsilon)}
\left[\psi(x)-\epsilon^2\right]\,dx\\
 & \leq & \norm{\psi}_1+4\epsilon^3-\epsilon^2(b-a)\\
 & < &  \norm{\psi}_1\quad\text{for small enough } \epsilon>0.
\end{eqnarray*}
And, if $g\in L^\infty([a,b])$ such that $\abs{g}>0$ almost
everywhere then $\norm{\psi_\epsilon g}_1<\norm{\psi g}_1$ for small
enough $\epsilon>0$.  A similar construction is used when $t$
equals $a$ or $b$.  Factoring $\psi$ shows the zeros of any
minimizing polynomial must be simple.

For $1<q<\infty$ use the same construction, with Taylor's theorem in 
\eqref{simpletaylor}, to get
\begin{eqnarray}
\norm{\psi_\epsilon}_q^q & = & \int_{x\in(t-\epsilon,t+\epsilon)}
\left[\epsilon^2-\psi(x)\right]^q\,dx+
\int_{x\not\in(t-\epsilon,t+\epsilon)}
\left[\psi(x)-\epsilon^2\right]^q\,dx\notag\\
 & \leq & \int_{x\in(t-\epsilon,t+\epsilon)}\epsilon^{2q}\,dx
+\int_{x\not\in(t-\epsilon,t+\epsilon)}
\psi^q(x)-q\epsilon^2\left[\psi(x)-\epsilon^2\right]^{q-1}\,dx
\label{simpletaylor}\\
 & < &  2\epsilon^{2q+1} +\norm{\psi}_q^q -q\epsilon^2
\int_{x\not\in(t-2\epsilon,t+2\epsilon)}
(3\epsilon^2)^{q-1}\,dx\notag\\
 & = & 2\epsilon^{2q+1} + 
\norm{\psi}_q^q-3^{q-1}q\epsilon^{2q}(b-a-4\epsilon)\notag\\
 & < & \norm{\psi}^q_q\quad\text{for small enough } \epsilon>0.\notag
\end{eqnarray}
As with the $q=1$ case above, it now follows that zeros of minimizing
polynomials must be simple.

(d) Suppose the minimum of $F_q$ occurs at 
both $\phi,\omega\in\polym$.
Let $\psi=(\phi+\omega)/2$.
Then $\psi\in\polym$  so
$\norm{\phi}_q\leq\norm{\psi}_q\leq (\norm{\phi}_q+\norm{\omega}_q)/2=
\norm{\phi}_q$.  The Minkowski inequality must then reduce to equality.
For $1<q<\infty$ this means
$\phi=d\omega$ for some $d>0$ \cite[p.~47]{liebloss}.
Since $\phi,\omega\in \polym$ we must have
$d=1$.
If $q=1$ then there is equality in the Minkowski
inequality if and only if $\phi\omega\geq 0$.  If there is equality, 
then $\phi$ and $\psi$ must share roots of odd multiplicity.  But by (c),
$\phi$ and $\psi$ have $m$ simple zeros in $[a,b]$.  Hence, $\phi=\psi$.
For $q=\infty$, Lemma~\ref{lemma2} shows there are 
$a< x_1<\cdots<x_{m-1}< b$,
$a< y_1<\cdots<y_{m-1}< b$
and
$a< z_1<\cdots<z_{m-1}< b$ such that
$\abs{\phi(x_i)}=\abs{\omega(y_i)}=\abs{\psi(z_i)}=\norm{\phi}_\infty$,
$\phi'(x_i)=\omega'(y_i)=\psi'(z_i)=0$ for each $1\leq i\leq m-1$.
Let $M_\phi=\{x_i\}_{i=1}^{m-1}$, $M_\omega=\{y_i\}_{i=1}^{m-1}$, 
$M_\psi=\{z_i\}_{i=1}^{m-1}$.  Let $z\in M_\psi$ then $\norm{\phi}_\infty
=\abs{\psi(z)}=\abs{\phi(z)+\omega(z)}/2<\norm{\phi}_\infty$ unless $z\in M_\phi
\cap M_\omega$.  Hence, $M_\phi=M_\omega$.  Therefore, $\phi'(x_i)=\psi'(x_i)$
for each $1\leq i\leq m-1$.  And then $\phi=\omega$.

(e) By parts (a) and (b) we need only consider  $\phi\in\polym$ 
with $m$ real roots in $[a,b]$.  Let $t_i\in[a,b]$ for $1\leq i\leq m$
and define $\phi\in\polym$ by $\phi(x)=\prod_{i=1}^m(x-t_i)$.

Suppose $1\leq q<\infty$.  Define 
$G_q(t)=\int_a^b\prod_{i=1}^m\abs{x-t_i}^q\,dx$.   
Then $G_q\fn[a,b]^m\to\R$
attains its minimum over $[a,b]^m$ if and only if $F_q$ attains its 
minimum over $\polym$.
The set $[a,b]^m$ is compact in $\R^m$.  And, $G_q$ is continuous.  For,
suppose $t\in[a,b]^m$ and $s^{(k)}\in[a,b]^m$ for each $k\in\N$
such that $s^{(k)}\to t$ in the Euclidean norm.  We have
$$\int_a^b\prod_{i=1}^m\abs{x-s^{(k)}_i}^q\,dx\leq
\int_a^b\prod_{i=1}^m(b-a)^{q}\,dx=(b-a)^{mq+1}.
$$
By dominated convergence (for example, \cite[Theorem~7.2]{bartleelements}), 
$\lim_{k\to\infty}G_q(s^{(k)})=G_q(t)$
and $G_q$ is continuous.  Therefore, $G_q$ attains its minimum over $[a,b]^m$.

The case $q=\infty$ is similar, using $F_\infty(\phi)=\max_{x\in[a,b]}\prod
_{i=1}^m
\abs{x-t_i}$.

(f) Without loss of generality, $b=-a$.  Suppose $\psi\in\polym$ is the 
unique minimizer of $F_q$.  Let $\omega(x)=\psi(-x)$ if
$m$ is even and $\omega(x)=-\psi(-x)$ if $m$ is odd.  Then
$\omega\in\polym$ and $\norm{\omega}_q=
\norm{\psi}_q$.  Let $\zeta=(\psi+\omega)/2$.  Then $\zeta\in\polym$ and
is even if $m$ is even, odd if $m$ is odd.  Also,
$\norm{\zeta}_q\leq (\norm{\psi}_q+\norm{\omega}_q)/2=\norm{\psi}_q$.
Hence, $\zeta=\psi=\omega$.  But then $\psi$ is even or odd
as $m$ is even or odd.  If $\phi\in\polym$ is neither even nor odd
then $F_q(\psi)<F_q(\phi)$.
\end{proof}
\begin{lemma}\label{lemma2}
Suppose $\phi\in\polym$ is a minimum of $F_\infty$.  Then
$\phi(x)=\prod_{i=1}^m(x-t_i)$ for $a< t_1<t_2
<\cdots<t_m< b$.  For each $1\leq i\leq m-1$
there is
$\xi_i\in(t_i,t_{i+1})$ such that 
$\abs{\phi(\xi_i)}=\norm{\phi}_\infty$.
\end{lemma}

\begin{proof}
This follows from the $q=\infty$ case of Lemma~\ref{lemma}(c).
\end{proof}

\section{$f''$ Henstock--Kurzweil integrable}\label{sectionhk}
Let $H\!K([a,b])$ be the set of functions integrable in the
Henstock--Kurzweil sense on $[a,b]$.  See, for example, \cite{lee}.
Note that $L^s([a,b])\subsetneq L^r([a,b])\subsetneq H\!K([a,b])$ for
all $1\leq r<s\leq \infty$.  An example of a function
$f\in H\!K([0,1])$ that is not in any $L^p([0,1])$ space is $f=F'$
where $F(x)=x^2\sin(x^{-3})$ for $x\in(0,1]$ and $F(x)=0$.
In this section we use the method
of Theorem~\ref{theoremminimize} to choose a monic quadratic $\phi$
to minimize the error in the resulting corrected trapezoidal rule
when $f''\in H\!K([a,b])$.  We also consider  the case when $f'$ is
merely continuous and then $f''$ exists as a distribution.  For
all of these cases, it turns out that amongst corrected trapezoidal
rules \eqref{basic}, the trapezoidal rule itself minimizes the error.

If $f\in H\!K([a,b])$ then the
Alexiewicz norm of $f$ is
$\norm{f}=\sup_{x\in[a,b]}\abs{\int_a^x f(t)\,dt}$.  If $g\fn[a,b]\to\R$
is of bounded variation then $fg\in H\!K([a,b])$.  Integration by
parts shows that
$\abs{\intab f(x)g(x)\,dx}\leq \abs{\intab f}\abs{g(b)}+\norm{f}Vg$,
where $Vg$ is the variation of $g$.  This is a version of an
inequality known in the literature as the
Darst--Pollard--Beesack inequality.  See
\cite{darstpollard}, \cite{beesack}. However, it appears earlier in
\cite{rieszlivingston}.
It is proved for a more symmetric version of the Alexiewicz norm
for Henstock--Kurzweil integrals in
\cite[Lemma~24]{talvilafouriertransform}. 

Under the
Alexiewicz norm, $H\!K([a,b])$ is a normed linear space but is not
complete.  We will discuss integration in the completion later
in this section.

\begin{theorem}\label{theoremhk}
Let $c$ be the midpoint of $[a,b]$.
Let $f\fn[a,b]\to\R$ such that $f'\in C([a,b])$
and $f''\in H\!K([a,b])$.
Amongst all monic quadratic polynomials used to generate \eqref{basic},
taking $\phi(x)=(x-c)^2-(b-a)^2/4=(x-a)(x-b)$ minimizes the error $\abs{E(f)}$.
This gives the
quadrature formula
\begin{equation}
\intab f(x)\,dx=\frac{b-a}{2}\left[f(a)+f(b)\right]
+E(f),\label{optimalhk}
\end{equation}
where
\begin{equation}
\abs{E(f)}  \leq 
\norm{f''}(b-a)^{2}/4.
\label{optimalerrorhk}
\end{equation}
For a uniform partition given by $x_i=a+(b-a)i/n$, $0\leq i\leq n$,
the composite formula is the trapezoidal rule
\begin{equation}
\intab f(x)\,dx=\frac{b-a}{2n}\left[f(a)+2\sum_{i=1}^{n-1} f(x_i)+f(b)\right] +
E(f),\label{compositehk}
\end{equation}
where
$$
\abs{E(f)}  \leq
\frac{\norm{f''}(b-a)^{2}}{4n}.
$$
The coefficient of $\norm{f''}$ in the error bounds is the best 
possible.
\end{theorem}
\begin{proof}
From \eqref{basic} we have $\abs{E(f)}\leq (\abs{\intab f''(x)\,dx}\abs{\phi(b)}+
\norm{f''}V\phi)/2$.
Write $\phi(x)=(x-r)^2+s$.  Then $V\phi=2\int_a^b\abs{x-r}\,dx$.
This is minimized when $r=c$. 
Note that $\phi(b)=0$ if and only if $s=-(b-a)^2/4$.
The unique error-minimizing polynomial is then $\phi(x)=(x-c)^2-(b-a)^2/4
=(x-a)(x-b)$,
for which $V\phi=(b-a)^2/2$.  Hence, $\abs{E(f)}\leq \norm{f''}(b-a)^2/4$.
To show this estimate cannot be improved, let $(\psi_n)$ be a
delta sequence as in the proof of Corollary~\ref{corollary1}.
Define $f_n''(x)= \psi_n(x-a)-2\psi_n(x-c)+\psi_n(b-x)$.  Then
$\intab f_n''(x)\,dx=0$ and $\norm{f_n''}=1$.  Integrate to get
$f_n(x)=\int_a^x(x-t)f_n''(t)\,dt$, modulo a linear function.
Since $\phi$ is continuous, $\lim_{n\to\infty}\abs{\intab f_n''(x)\phi(x)\,dx}
=\abs{\phi(a)-2\phi(c)+\phi(b)}=(b-a)^2/2$.
\end{proof}

In the terminology of Theorem~\eqref{theoremminimize} the minimizing
polynomial has $\beta=1$ and $\alpha =(b-a)/2$ and
is the same as the one that generated the
trapezoidal rule in \eqref{trapp}.
Notice that the variation is additive over disjoint intervals so
the error in the composite formula is order $1/n$, compared with
error of order $1/n^2$ in 
Corollary~\ref{corollarycomposite} when
$f''\in L^p([a,b])$.

An equivalent norm is $\norm{f}^\ast=\sup_I\abs{\int_If(x)\,dx}$ where the
supremum is taken over all intervals $I\subset[a,b]$.  It is easy to
see that $\norm{f}\leq\norm{f}^\ast\leq 2\norm{f}$ for all $f\in H\!K([a,b])$. 
In terms of this norm, 
\eqref{optimalerrorhk} implies
$\abs{E(f)}  \leq 
\norm{f''}^\ast(b-a)^{2}/4$.  This inequality appears as Theorem~21
in \cite{dingyewang}.
We improve this inequality by a factor of $1/2$ and show our result is sharp.
\begin{corollary}
With the assumptions of Theorem~\ref{theoremhk} we have
$\abs{E(f)}\leq (b-a)^2\norm{f''}^\ast/8$.  The constant $1/8$ is the
best possible.
\end{corollary}

\begin{proof}
Use the second mean value theorem for integrals \cite{celidze}.  If
$\phi$ is monotonic on $[a,b]$ then there is $\xi\in[a,b]$ such that
$$
E(f)  =  \frac{1}{2}\intab f''(x)\phi(x)\,dx
  =  \frac{\phi(a)}{2}\int_a^\xi f''(x)\,dx +
\frac{\phi(b)}{2}\int_\xi^b f''(x)\,dx.
$$
Then $\abs{E(f)}\leq (\abs{\phi(a)}+\abs{\phi(b)})\norm{f''}^\ast/2$.
This is minimized by taking $\phi(x)=(x-a)^2$ or
$\phi(x)=(x-b)^2$, for which $\abs{E(f)}\leq (b-a)^2\norm{f''}^\ast/2$.  If
$\phi$ has a minimum at  $r\in(a,b)$ then there are $\xi_1\in[a,r]$ and
$\xi_2\in[r,b]$ such that
\begin{eqnarray*}
E(f) & = & \frac{\phi(a)}{2}\int_a^{\xi_1} f''(x)\,dx +
\frac{\phi(r)}{2}\int_{\xi_1}^r f''(x)\,dx
+\frac{\phi(r)}{2}\int_r^{\xi_2} f''(x)\,dx
+\frac{\phi(b)}{2}\int_{\xi_2}^b f''(x)\,dx.
\end{eqnarray*}
It follows that $\abs{E(f)}\leq (\abs{\phi(a)}+\abs{\phi(r)}
+\abs{\phi(b)})\norm{f''}^\ast/2$.  Choosing $\phi(x)=(x-a)(x-b)$ minimizes
the coefficient of $\norm{f''}^\ast$ and we get
$\abs{E(f)}\leq (b-a)^2\norm{f''}^\ast/8$.
To prove this estimate is sharp, let $(\psi_n)$ be a
delta sequence as in the proof of Corollary~\ref{corollary1}.
Define $f_n''(x)= \psi_n(x-c)$.  Then $\norm{f_n''}^\ast=1$ and
$\abs{E(\psi_n)}\to (c-a)(b-c)/2=(b-a)^2/8$.
\end{proof}

The completion of $H\!K([a,b])$ in the Alexiewicz norm is the Banach space
$\alexc$. Each element of $\alexc$ is the distributional derivative
of a function in $\balexc=\{F\in C([a,b])\mid F(a)=0\}$.  Note that
$\balexc$ is a Banach space under usual pointwise operations and the
uniform norm.  If $f\in\alexc$ then there is a unique primitive $F\in\balexc$
such that $F'=f$.  The distributional derivative is $\langle F',\phi\rangle
=-\langle F,\phi'\rangle=-\intab F(x)\phi'(x)\,dx$ where 
$\phi\in C_c^\infty(\R)$.  The
Alexiewicz norm of $f$ is $\norm{f}=\sup_{x\in[a,b]}\abs{\int_a^x f}=
\norm{F}_\infty$.  This makes
$\alexc$ into a Banach space isometrically isomorphic to $\balexc$.
The {\it continuous primitive integral} of $f\in\alexc$ is then
$\intab f=F(b)-F(a)$.  If $g$ is of bounded variation then the
integration by parts formula is given in terms of a Riemann--Stieltjes
integral as $\intab fg=F(b)g(b)-\intab F(x)\,dg(x)$.  Note that
$\alexc$ contains $H\!K([a,b])$ and hence $L^p([a,b])$ for each 
$1\leq p\leq\infty$.  If $F$ is a continuous function such that
its pointwise derivative $F'(x)=0$ almost everywhere then the Lebesgue
integral of $F'(x)$ exists and is $0$ but the continuous primitive
integral is $\intab F'=F(b)-F(a)$.  If $F$ is continuous such that the pointwise
derivative exists nowhere then the Lebesgue integral of $F'(x)$ does
not exist but $F'\in\alexc$ and $\intab F'=F(b)-F(a)$.   
See \cite{talviladenjoy}
for details.  Note that  if $F\in C([a,b])$ then $F'\in\alexc$ and
$\intab F'=F(b)-F(a)$.  An advantage of the continuous primitive integral
is that the space of primitives is simple.  It is
the continuous functions while for the Henstock--Kurzweil integral it is
a complicated space called $ACG^\ast$.  See \cite{celidze} for the definition.
\begin{corollary}\label{corollaryalexc}
Let $f\in C^1([a,b])$.  Then $f''\in\alexc$ and the formulas in
Theorem~\ref{theoremhk}
hold.
\end{corollary}

In a sense this now reduces to estimates on $f'$ since the Alexiewicz
norm of $f''$ is the uniform norm of $f'$.
The formulas in Theorem~\ref{theoremhk} also hold when $f'$ is a
regulated function.  This is a function that has a left limit and
a right limit at each point.  See \cite{talvilaregulated} for details.

\section{Exact for cubics}\label{sectionexact}
In this section we show \eqref{basic} is exact for all $\phi$ 
when $f$ is a linear
function.  We also show \eqref{basic} is exact 
for all cubic polynomials $f$ whenever $\phi(x)=(x-c)^2-(b-a)^2/12$.
\begin{theorem}
Let $f\fn[a,b]\to\R$ and let $c$ be the midpoint of $[a,b]$. 
Let $\phi$ be a monic quadratic.  Write
\begin{equation}
\intab f(x)\,dx=\frac{1}{2}\left[-f(a)\phi'(a)+f(b)\phi'(b)+
f'(a)\phi(a)-f'(b)\phi(b)\right]
+E(f).\label{exact}
\end{equation}
(a) If $f$ is a linear function then $E(f)=0$ for all such $\phi$.
(b) $E(f)=0$ for all cubic polynomials $f$ if and only if
$\phi(x)=(x-c)^2-(b-a)^2/12$.
\end{theorem}
\begin{proof}
The proof of (a) is straightforward using 
$E(f)=(1/2)\intab f''(x)\phi(x)\,dx$.  
By (a) we need only consider $f(x)=Ax^3+Bx^2$.  Write $\phi(x)=x^2+Cx+D$.
The equation $\intab(6Ax+2B)(x^2+Cx+D)\,dx=0$ gives the linear system
\begin{eqnarray}
2(a^2+ab+b^2)C+3(a+b)D & = & -\frac{3}{2}(a+b)(a^2+b^2)\\
(a+b)C +2D & = & -\frac{2}{3}(a^2+ab+b^2).\label{cubic2}
\end{eqnarray}
The solution is $C=-(a+b)$, $D=(a^2+4ab+b^2)/6$, and this gives
$\phi(x)=(x-c)^2-(b-a)^2/12$.
\end{proof}
Note that this is the same $\phi$ as in Corollary~\ref{corollary2}.
Also, $E(f)=0$ for all quadratic polynomials $f$ if and only if \eqref{cubic2}
holds.

\end{document}